\documentclass[reqno]{amsart}

\calclayout

\usepackage[unicode]{hyperref}
\usepackage[lite]{amsrefs}
\usepackage{amsmath,amssymb}
\usepackage{graphicx}
\usepackage{xspace}
\let\zeroslash\emptyset

\def\paragraph{\subsection}
\makeatletter
\def\@sect#1#2#3#4#5#6[#7]#8{%
  \edef\@toclevel{\ifnum#2=\@m 0\else\number#2\fi}%
  \ifnum #2>\c@secnumdepth \let\@secnumber\@empty
  \else \@xp\let\@xp\@secnumber\csname the#1\endcsname\fi
  \@tempskipa #5\relax
  \ifnum #2>\c@secnumdepth
    \let\@svsec\@empty
  \else
    \refstepcounter{#1}%
    \edef\@secnumpunct{%
      \ifdim\@tempskipa>\z@ % not a run-in section heading
        \@ifnotempty{#8}{.\@nx\enspace}%
      \else
        \@ifempty{#8}{.}{.\@nx\enspace}%
      \fi
    }%
      \ifnum #2=\tw@ \def\@secnumfont{\bfseries}\fi{}%
    \protected@edef\@svsec{%
      \ifnum#2<\@m
        \@ifundefined{#1name}{}{%
          \ignorespaces\csname #1name\endcsname\space
        }%
      \fi
      \@seccntformat{#1}%
    }%
  \fi
  \ifdim \@tempskipa>\z@ % then this is not a run-in section heading
    \begingroup #6\relax
    \@hangfrom{\hskip #3\relax\@svsec}{\interlinepenalty\@M #8\par}%
    \endgroup
    \ifnum#2>\@m \else \@tocwrite{#1}{#8}\fi
  \else
  \def\@svsechd{#6\hskip #3\@svsec
    \@ifnotempty{#8}{\ignorespaces#8\unskip
       \@addpunct.}%
    \ifnum#2>\@m \else \@tocwrite{#1}{#8}\fi
  }%
  \fi
  \global\@nobreaktrue
  \@xsect{#5}}
\makeatother

\makeatletter
\def\pxspace{\@ifnextchar.{\@}{.\@\xspace}}
\makeatother

\makeatletter
\def\citeHerr%
{\@ifnextchar*{\citeHerrs}{\cite{Herr1973a}\allowbreak\cite{Herr1973a-en}}}
\def\citeHerrs*#1#2{\cite{Herr1973a}*{#1}\allowbreak\cite{Herr1973a-en}*{#2}}
\makeatother

\newcommand{\hatzero}{\hat{0}}
\newcommand{\hatone}{\hat{1}}
\newcommand{\mvert}{\,|\,}      % "such that":  \{ x \in S \mvert x > y \}

\newtheorem{theorem}{Theorem}[section]

\newtheorem{definition}[theorem]{Definition}
\newtheorem{lemma}[theorem]{Lemma}

\newcommand{\disconnect}{\leavevmode\par}

\newcommand{\forwardref}[1]{\ref{#1}}

\makeatletter
\def\pxspace{\@ifnextchar.{\@}{.\@\xspace}}
\makeatother

\makeatletter
\newcommand{\manuallabel}[2]{\edef\@currentlabel{#2}\label{#1}}
\makeatother
\newcommand{\itemlabel}[2]{\item[#1]\manuallabel{#2}{#1}}

\newcommand{\interject}[1]%
{\noalign{\begin{quote}#1\end{quote}}}

\newcommand{\relgamma}{\mathrel{\gamma}}
\newcommand{\relrho}{\mathrel{\rho}}
\newcommand{\relxi}{\mathrel{\xi}}

\newcommand{\ccolon}{\mathrel{::}}
\newcommand{\suppl}{+}
\newcommand{\supmi}{--}	% an en-dash is about the width of a +

\newcommand{\lffc}{l.f.f.c\pxspace}
\newcommand{\fmc}{f.m.c\pxspace}

\renewcommand{\P}[2]{{\phi_{#1}^{#2}}} 		% \phi_x^y
\newcommand{\PI}[2]{{{\phi_{#1}^{#2}}^{-1}}}       % {\phi_x^y}^{-1}
\newcommand{\F}[2]{{F_{#1}^{#2}}}                 % F_x^y
\newcommand{\I}[2]{{I_{#1}^{#2}}}                 % I_x^y
\newcommand{\W}[2]{{\omega_{#1}^{#2}}}            % \omega_x^y
\DeclareMathOperator{\id}{id}
\DeclareMathOperator{\im}{im}
\DeclareMathOperator{\dom}{dom}

\begin{document}

% Must be after \begin{document}.
\ifdefined\draftlabels
  \global\let\labelorig\label
  \global\def\label#1{[label:#1]\ \labelorig{#1}}
\fi

\title[Modular lattices --- gluing axioms]%
{On the structure of modular lattices --- Axioms for gluing}
\author{Dale R. Worley}
\email{worley@alum.mit.edu}
\date{Apr 7, 2025} % format is Mmm dd, yyyy.

\begin{abstract}
This paper explores alternative statements of the axioms for
lattice gluing, focusing on
lattices that are modular, locally finite, and
have finite covers, but may have infinite height.
We give a set of ``maximal'' axioms that
maximize what can be immediately adduced about the
structure of a valid gluing.
We also give a set of ``minimal'' axioms that minimize
what needs to be adduced to prove that a system of blocks is a valid
gluing. This system appears to be novel in the literature.
A distinctive feature of the minimal axioms is that they
involve only relationships between elements of the skeleton
which are within an interval $[x \wedge y, x \vee y]$ where either $x$
and $y$ cover $x \wedge y$ or they are covered by $x \vee y$.  That
is, they have a decidedly \emph{local} scope, despite that the
resulting sum lattice, being modular, has \emph{global} structure,
such as the diamond isomorphism theorem.
\end{abstract}

\maketitle

\tableofcontents
%\listoffigures

\section{Introduction}

This paper is to explore alternative statements of the axioms for
lattice gluing.  Lattice gluing is described in
Herrmann \citeHerr, Day \cite{DayHerr1988a}, and
\cite{Wor2025a}.  In this paper, we focus on the variant in
\cite{Wor2025a}, viz. lattices that are modular, locally finite, and
have finite covers, but may have infinite height.
We treat monotony (that blocks are incomparable
as subsets of the lattice) as an independent property in order to
support gluings that are ``polytone''.  (Polytone gluings can arise from
considering sublattices of dissected lattices, among other situations.)

We also treat the existence of $\hatzero$ in the skeleton as an
independent property.
However, for all these sets of axioms, $\hatzero$ exists in the
skeleton iff $\hatzero$ exists in the sum
lattice,\cite{Wor2025a}*{Th.~4.70}
so we will not discuss this property further.

We have two goals.  One is to determine a set of ``maximal'' axioms,
in order to maximize what can be immediately adduced about the
structure of a valid gluing.  These axioms, or statements close to
them, are in the prior literature.

The second goal is to determine a set of ``minimal'' axioms, to minimize
what needs to be adduced to prove that a system of blocks is a valid
gluing. Some of these axioms appear to be novel in the literature.
A distinctive feature of the minimal axioms is that they
involve only relationships between elements of the skeleton
which are within an interval $[x \wedge y, x \vee y]$ where either $x$
and $y$ cover $x \wedge y$ or they are covered by $x \vee y$.  That
is, they have a decidedly \emph{local} scope, despite that the
resulting sum lattice, being modular, has \emph{global} structure,
such as the diamond isomorphism theorem.

The minimal axioms will reduce the work needed to classify valid
gluings.

\section{Preliminaries}

\paragraph{Lattice properties}

\begin{definition} \label{def:lffc}
We define that
a lattice is \emph{locally finite with finite covers} ---
abbreviated \emph{\lffc} --- if it is both locally finite and has
finite covers (both upper and lower).
\end{definition}

\begin{definition} \label{def:fmc}
We abbreviate that a lattice is finite, modular, and
complemented by saying it is \emph{\fmc.}
\end{definition}

\paragraph{Tolerances}
We use Bandelt \cite{Band1981a} as a reference for tolerances on
lattices.

\begin{definition} \cite{Band1981a}*{sec.~1}
We define that
a reflexive and symmetric binary relation $\relxi$ on a lattice $L$ is
a \emph{tolerance} if $\xi$ is compatible with the
meet and join of $L$, that is,
\begin{enumerate}
\item $a \relxi b$ and $c \relxi d$ imply $a \vee c \relxi b \vee d$, and
\item $a \relxi b$ and $c \relxi d$ imply $a \wedge c \relxi b \wedge d$.
\end{enumerate}
\end{definition}

\begin{lemma} \label{lem:tol} \cite{Band1981a}*{Lem.~1.1}
If $\relxi$ is a tolerance on a lattice $L$, then:
\begin{enumerate}
\item \label{lem:tol:i1} $x \relxi z$ and $x \leq y \leq z$ imply
$x \relxi y$ and $y \relxi z$.
\item \label{lem:tol:i2} $x \relxi y$ iff $x \vee y \relxi x \wedge y$,
\item \label{lem:tol:i3} $t \relxi x$, $t \relxi y$, and $t \leq x \wedge y$
imply $t \relxi x \vee y$.
\item \label{lem:tol:i4} $t \relxi x$, $t \relxi y$, and $t \geq x \vee y$
imply $t \relxi x \wedge y$.
\item \label{lem:tol:i5} $x \relxi x \vee y$ and $y \relxi x \vee y$
imply $x \vee y \relxi x \wedge y$, $x \relxi y$, and
$x \wedge y \relxi x, y$, that is, $\relxi$ holds among every pair of
$x$, $y$, $x \vee y$, and $x \wedge y$.
\item \label{lem:tol:i6} $x \relxi x \wedge y$ and $y \relxi x \wedge y$
imply $x \wedge y \relxi x \vee y$, $x \relxi y$, and
$x \vee y \relxi x, y$, that is, $\relxi$ holds among every pair of
$x$, $y$, $x \wedge y$, and $x \vee y$.
\end{enumerate}
\end{lemma}
\begin{proof} \disconnect
Regarding (\ref{lem:tol:i1}):
Since $x \relxi z$ and $y \relxi y$, $x \vee y \relxi z \vee y$, which
is equivalent to $y \relxi z$.
Similarly, $x \wedge y \relxi z \wedge y$, which
is equivalent to $x \relxi y$.

Regarding (\ref{lem:tol:i2}), (\ref{lem:tol:i3}), and (\ref{lem:tol:i4}):
see Bandelt \cite{Band1981a}*{Lem.~1.1(1--3)~Proof}.

Regarding (\ref{lem:tol:i5}):
By (\ref{lem:tol:i4}), $x \vee y \relxi x \wedge y$.
By (\ref{lem:tol:i2}), $x \relxi y$.
By (\ref{lem:tol:i1}), $x \wedge y \relxi x, y$.

Regarding (\ref{lem:tol:i6}):
Proved dually to (\ref{lem:tol:i5}).
\end{proof}

\begin{definition} \label{def:leq-gamma}
If $\relxi$ is a tolerance on a lattice $L$, for $x, y \in L$, we define
$x \leq_\xi y$ iff $x \leq y$ and $x \relxi y$.
We define $x \geq_\xi y$ iff $x \geq y$ and $x \relxi y$.
\end{definition}

\paragraph{Ordered tolerances}
We now show that a tolerance on a lattice is characterized by its values on
pairs of comparable elements.%
\footnote{Pairs of comparable elements are often called
\emph{quotients}.}%
\footnote{Curiously, although it is well-known
(\cite{Birk1967a}*{\S II.4}) that
congruences on lattice elements are determined by the congruence on
quotients, that tolerances are determined by the tolerance on
quotients is not mentioned in \cite{Band1981a}.
Although that paper does mention ``For instance, every compatible
relation containing
the partial order $\leq$ corresponds to a tolerance relation,
and \emph{vice versa}''
and our lem.~\forwardref{lem:ord-tol} is ``every compatible
relation \emph{contained in}
the partial order $\leq$ corresponds to a tolerance relation
and vice versa.''}

\begin{definition} \label{def:ord-tol}
We call a reflexive relation $\relxi$ ($x \relxi y$) on pairs of comparable
elements ($x \leq y$) of a lattice
an \emph{ordered tolerance} iff it is compatible with the meet and
join of the lattice, that is,
if $x \relxi x^\prime$ and $y \relxi y^\prime$
(and thus $x \leq x^\prime$ and $y \leq y^\prime$),
then $x \wedge y \relxi x^\prime \wedge y^\prime$
and $x \vee y \relxi x^\prime \vee y^\prime$.
\end{definition}

\begin{lemma} \label{lem:ord-tol-convex}
An ordered tolerance $\relxi$, like a tolerance, is
``convex'', in that if $x \leq y \leq z$ and $x \relxi z$,
then $x \relxi y$ and $y \relxi z$.
\end{lemma}
\begin{proof} \disconnect
Since $y \relxi y$, by compatibility,
$x = x \wedge y \relxi z \wedge y = y$ and
$y = x \vee y \relxi z \vee y = z$.
\end{proof}

\begin{lemma} \label{lem:ord-tol}
A tolerance on a lattice restricted to
pairs of comparable elements of the lattice is a (unique) ordered
tolerance.
An ordered tolerance on a lattice can be extended to a
tolerance on the lattice, which is unique.
\end{lemma}
\begin{proof} \disconnect
The first statement is trivial.

To prove the second statement, consider
an ordered tolerance $\relrho$.
Define the
relationship $\relxi$ on all pairs of elements of the lattice by
$x \relxi y$ iff $x \wedge y \relrho x \vee y$.
Trivially, $\relxi$ extends $\relrho$.
By lem.~\ref{lem:tol}(\ref{lem:tol:i2}),
any tolerance which extends $\relrho$ must equal $\relxi$,
so it remains to prove that $\relxi$ is always a tolerance.

We show that $\relxi$ is compatible with join as follows:
Let $x, y, x^\prime, y^\prime$ be in the lattice and
$x \relxi y$ and $x^\prime \relxi y^\prime$.
Then $x \wedge y \relrho x \vee y$ and
$x^\prime \wedge y^\prime \relrho x^\prime \vee y^\prime$.
Since $\relrho$ is compatible,
\begin{equation*}
(x \wedge y) \vee (x^\prime \wedge y^\prime) \relrho
(x \vee y) \vee (x^\prime \vee y^\prime) =
(x \vee x^\prime) \vee (y \vee y^\prime). \tag{a}
\end{equation*}
Trivially,
$x \wedge y \leq x \vee x^\prime$,
$x \wedge y \leq y \vee y^\prime$,
$x^\prime \wedge y^\prime \leq x \vee x^\prime$, and
$x^\prime \wedge y^\prime \leq y \vee y^\prime$.
Combining these gives
\begin{equation*}
(x \wedge y) \vee (x^\prime \wedge y^\prime) \leq
(x \vee x^\prime) \wedge (y \vee y^\prime) \leq
(x \vee x^\prime) \vee (y \vee y^\prime). \tag{b}
\end{equation*}
By convexity, (a) and (b) imply
$$ (x \vee x^\prime) \wedge (y \vee y^\prime) \relrho
(x \vee x^\prime) \vee (y \vee y^\prime) $$
which means
$$ x \vee x^\prime \relxi y \vee y^\prime. $$
Thus $\relxi$ is compatible with join.

Dually, we show that $\relxi$ is compatible with meet.
Thus $\relxi$ is a tolerance.
\end{proof}

\paragraph{Partial bijections}
We will be frequently using partial bijections between two sets.
Note that the class of partial bijections (like bijections, but unlike
partial functions or ordinary
functions) is completely self-dual under the operation of inverse;
a mapping is a partial bijection iff its inverse is a partial
injection.

\begin{definition}
A \emph{partial bijection} $f$ from a \emph{source} set $S$ to
a \emph{target} set $T$ is an
injection from a (possibly empty) subset $\dom f$ of $S$ to a subset
$\im f$ of $T$.  A partial bijection is necessarily a bijection from $\dom f$
to $\im f$.
A partial bijection is defined to be \emph{empty} iff
$\dom f = \im f = \zeroslash$.
\end{definition}

\begin{definition}
For partial bijections, we use
the common notations for functions, but replacing ``$:$'' with
``$\ccolon$''.
$$ f \ccolon S \rightarrow T $$
denotes that $f$ is a partial bijection with source (not domain) $S$
and target (not image) $T$.
$$ f \ccolon x \mapsto \phi(x) $$
denotes that $f$ maps $x$ to $\phi(x)$, where the source, target,
domain, and image are given by the context.
$$ f \ccolon S \rightarrow T \ccolon x \mapsto \phi(x) $$
denotes that $f$ is a partial bijection with source $S$
and target $T$, but the domain is still given by the
context, usually the set of values $x$ for which the formula $\phi(x)$
is well-defined.
\end{definition}

\begin{lemma}
If $f \ccolon S \rightarrow T$, then $f^{-1} \ccolon T \rightarrow S$.
That is, the inverse of a partial bijection is a partial bijection.
$\dom f^{-1} = \im f$ and $\im f^{-1} = \dom f$.
\end{lemma}

\begin{definition} \label{def:circ}
We use ``$\circ$'' to denote the composition of partial bijections,
that is if $f \ccolon S \rightarrow T$ and
$g \ccolon T \rightarrow U$,
$g \circ f \ccolon S \rightarrow U \ccolon x \mapsto g(f(x))$,
where $g \circ f$ is defined only for those
elements $x$ of $S$ for which $f$ is defined \emph{and} for which
$g$ is defined on $f(x)$.
\end{definition}

\begin{lemma}
Given $f \ccolon S \rightarrow T$ and
$g \ccolon T \rightarrow U$,
\begin{enumerate}
\item $\im g \circ f = g(\im f)$,
\item $\dom g \circ f = f^{-1}(\dom g)$, and
\item $(g \circ f)^{-1} = f^{-1} \circ g^{-1}$.
\end{enumerate}
\end{lemma}

\begin{definition}
$\id_S$ denotes the identity map on the set $S$, which is also a
partial bijection with source and target $S$.
We state that a partial bijection $f$ \emph{equals} $\id_S$ when
$\dom f = \im f = S$ and $f(x) = x$ for every $x \in S$,
even if the source and/or
target of $f$ is larger than $S$.
\end{definition}

\section{Baseline axioms}

We start with a baseline set of gluing axioms copied
from \cite{Wor2025a}.  See that paper for the relationships between the
baseline axioms and the definitions used by Herrmann and Day.
We update the terminology by calling the glued system ``monotone'' if
it satisfies the full set of axioms --- following Herrmann --- but we
consider the ``monotony'' property to be separable from the rest of
the axioms; systems that satisfy the axioms other than monotony are
called ``polytone''.%
\footnote{We choose the term ``polytone'' solely because it formally
is an antonym of ``monotone''.}

Our numbering of axioms largely follows the numbering of
the axioms in \cite{Wor2025a}, changing ``MC'' (for
``modular connected'') to ``PS'' (for ``polytone system'').
We split (MC7) into (PS7) and (PS8),
update (MC8.1) to (PS9), and update (MC8.2) to (MS) (for ``monotone
system'').
We use ``$\dom \P{x}{y}$'' and ``$\im \P{x}{y}$'' rather than
the equivalent $\F{x}{y}$ and $\I{x}{y}$.
We consider the $\P{x}{y}$ as partial bijections from $L_x$ to $L_y$
rather than as bijections from $\F{x}{y} \subset L_x$ to
$\I{x}{y} \subset L_y$.
This revised definition of $\P{\bullet}{\bullet}$, combined with our
definition of ``$\circ$'' for partial bijections (def.~\ref{def:circ}),
has no substantive effect but
allows \forwardref{PS6} = (MC6) to be stated more simply.

\begin{definition} \label{def:polytone}
The \emph{baseline axioms}:
We define a \emph{polytone system} to be comprised of:
\begin{enumerate}
\item[\textbullet] a \emph{skeleton} lattice $S$,
\item[\textbullet] an \emph{overlap tolerance} $\relgamma$, which is
a tolerance on $S$,
\item[\textbullet] a family of \emph{blocks} $(L_x)_{x \in S}$, which
are lattices, and
\item[\textbullet] a family of \emph{connections}
$(\P{x}{y})_{x, y \in S, x \leq_\gamma y}$, which are mappings,
\end{enumerate}
that satisfies these axioms:
\begin{enumerate}
\itemlabel{(PS1)}{PS1} The skeleton $S$ is \lffc.
\itemlabel{(PS2)}{PS2} The blocks $L_x$ are \fmc.
\itemlabel{(PS3)}{PS3}
For $x \relgamma y$ in $S$,
each $\P{x}{y}$ is a partial bijection from source $L_x$ to target $L_y$,
$\dom \P{x}{y}$ is a filter of $L_x$,
$\im \P{x}{y}$ is an ideal of $L_y$,
and $\P{x}{y}$ is a lattice isomorphism.
\itemlabel{(PS4)}{PS4} For any $x \in S$, $\P{x}{x} = \id_{L_x}$.
\itemlabel{(PS5)}{PS5} If $x \leq_\gamma y \leq_\gamma z$ in $S$ and
$\im \P{x}{y} \cap \dom \P{y}{z} \neq \zeroslash$, then $x \relgamma z$.
\itemlabel{(PS6)}{PS6} For every $x \leq z \leq y$ in $S$ where
$x \relgamma y$, then $\P{x}{y} = \P{z}{y} \circ \P{x}{z}$.
\itemlabel{(PS7)}{PS7} For every $x, y \in S$ for which
$x \relgamma y$,
$\im \P{x}{x \vee y} \cap \im \P{y}{x \vee y} \subset
\im \P{x \wedge y}{x \vee y}$.
\itemlabel{(PS8)}{PS8} For every $x, y \in S$ for which
$x \relgamma y$,
$\dom \P{x \wedge y}{x} \cap \dom \P{x \wedge y}{y} \subset
\dom \P{x \wedge y}{x \vee y}$.
\itemlabel{(PS9)}{PS9} If $x \lessdot y$ in $S$, then $x \relgamma y$.
\end{enumerate}
\end{definition}

\begin{definition}
We define a \emph{monotone system} to be a polytone system that in
addition satisfies:
\begin{enumerate}
\itemlabel{(MS)}{MS} If $x \lessdot y$ in $S$,
$\dom \P{x}{y} \neq L_x$ and $\im \P{x}{y} \neq L_y$.
\end{enumerate}
\end{definition}

Our first change is to extend the definition of $\P{x}{y}$ to all
$x \leq y$ in $S$, even if $x \not\relgamma y$:

\begin{definition} \label{def:phi-extended}
We extend the definition of $\P{x}{y}$ to all
$x \leq y$ in $S$, even if $x \not\relgamma y$:
If $x \leq y$ in $S$ and $x \not\relgamma y$, we define
$\P{x}{y}$ to be the empty partial bijection with source $L_x$ and
target $L_y$.
\end{definition}

This extension does not affect satisfaction of the axioms
because the axioms never depend on $\P{x}{y}$ unless $x \relgamma y$.
This change will be used in both the minimal and maximal axioms.

\section{Minimal axioms}

Now we weaken the set of statements we must adduce to prove a polytone
system.  Rather than showing our efforts step-by-step,
we will give the end result of our efforts as a definition,
then work through the proof that it is equivalent to the baseline
axioms.

\paragraph{The axioms} \disconnect

\begin{definition}
Given a system of maps $(\W{x}{y})_{x,y \in S, x \lessdot y}$
where each
$\W{x}{y}$ is a partial bijection from source $L_x$ to target $L_y$,
we define, for any saturated chain
$x_0 \lessdot x_1 \lessdot \cdots \lessdot x_n$ of length $n \geq 0$,
$\W{x_\bullet}{*}$ = $\W{x_{n-1}}{x_n} \circ \cdots
\circ \W{x_1}{x_2} \circ \W{x_0}{x_1}$.
(If $n = 0$, then $\W{x_\bullet}{*} = \id_{L_{x_0}}$.)
\end{definition}

The names of the minimal axioms are distinguished by including ``\supmi''.
The axiom numbers are non-sequential, so that when
corresponding axiom numbers exist in the baseline and minimal axioms,
the axioms are similar.

\begin{definition} \label{def:polytone-min}
The \emph{minimal axioms}:
We define a \emph{polytone system} to be comprised of:
\begin{enumerate}
\item[\textbullet] a \emph{skeleton} lattice $S$,
\item[\textbullet] a family of \emph{blocks} $(L_x)_{x \in S}$, which
are lattices, and
\item[\textbullet] a family of \emph{connections}
$(\W{x}{y})_{x, y \in S, x \lessdot y}$, which are mappings,
\end{enumerate}
that satisfies these axioms:
\begin{enumerate}
\itemlabel{(PS1\supmi)}{PS1-} The skeleton $S$ is \lffc.
\itemlabel{(PS2\supmi)}{PS2-} The blocks $L_x$ are \fmc.
\itemlabel{(PS3\supmi)}{PS3-}
For $x \lessdot y$ in $S$,
each $\W{x}{y}$ is a partial bijection from source $L_x$ to target $L_y$,
$\dom \W{x}{y}$ is a filter of $L_x$,
$\im \W{x}{y}$ is an ideal of $L_y$,
and $\W{x}{y}$ is a lattice isomorphism.
\itemlabel{(PS6\supmi)}{PS10-} Either:
\begin{enumerate}
\itemlabel{(PS6\supmi$\wedge$)}{PS10-m}
For any $x \wedge y \lessdot x, y$ in $S$,
there exist saturated chains
$x = x_0 \lessdot x_1 \cdots \lessdot x_n = x \vee y$
and
$y = y_0 \lessdot y_1 \cdots \lessdot y_m = x \vee y$
such that $\W{x_\bullet}{*} \circ \W{x \wedge y}{x} =
\W{y_\bullet}{*} \circ \W{x \wedge y}{y}$, or
\itemlabel{(PS6\supmi$\vee$)}{PS10-j}
For any $x, y \lessdot x \vee y$ in $S$,
there exist saturated chains
$x \wedge y = x_0 \lessdot x_1 \cdots \lessdot x_n = x$
and
$x \wedge y = y_0 \lessdot y_1 \cdots \lessdot y_m = y$
such that $\W{x}{x \vee y} \circ \W{x_\bullet}{*} \circ  =
\W{y}{x \vee y} \circ \W{y_\bullet}{*}$.
\end{enumerate}
\itemlabel{(PS7\supmi$\wedge$)}{PS11-mm}
For any $x \wedge y \lessdot x, y$ in $S$,
there exists a saturated chain
$x \wedge y = z_0 \lessdot z_1 \cdots \lessdot z_n = x \vee y$
such that
$\dom \W{x \wedge y}{x} \cap \dom \W{x \wedge y}{y} \subset
\dom \W{z_\bullet}{*}$.
\itemlabel{(PS8\supmi$\vee$)}{PS11-jj}
For any $x, y \lessdot x \vee y$ in $S$,
there exists a saturated chain
$x \wedge y = z_0 \lessdot z_1 \cdots \lessdot z_n = x \vee y$
such that
$\im \W{x}{x \vee y} \cap \im \W{y}{x \vee y} \subset
\im \W{z_\bullet}{*}$.
\end{enumerate}
\end{definition}

The minimal axioms \ref{PS1-} and \ref{PS2-} are the same as baseline
axioms \ref{PS1} and \ref{PS2}.
\ref{PS3-} is a weaker form of \ref{PS3},
\ref{PS10-} is a weaker form of \ref{PS6},
\ref{PS11-mm} is a weaker form of \ref{PS7}, and
\ref{PS11-jj} is a weaker form of \ref{PS8}.

The additional axiom for a monotone system, (MS), becomes
\begin{enumerate}
\item[(MS)] \textit{If $x \lessdot y$ in $S$,
$\dom \W{x}{y} \neq L_x$, and $\im \W{x}{y} \neq L_y$.}
\end{enumerate}

The axiom (\ref{PS10-} has two alternative
forms which are each other's duals.  Thus there are actually two
sets of axioms, which we will show are equivalent.
The axioms with the antecedent ``$x \wedge y \lessdot x, y$''
have $\wedge$ differentiating their names, and
the axioms with the antecedent ``$x, y \lessdot x \vee y$''
have $\vee$ differentiating their names.

If $S$ is modular, then $x \wedge y \lessdot x, y$ iff
$x, y \lessdot x \vee y$ and so
\ref{PS10-m} becomes equivalent to \ref{PS10-j}.  However, $S$ is
often not modular.\citeHerr*{Satz~7.2}{Th.~7.2}.

\paragraph{Defining \texorpdfstring{$\P{\bullet}{\bullet}$}{\83\306}}
\disconnect

\begin{lemma} \label{lem:concat}
Given a saturated chain
$x_0 \lessdot x_1 \lessdot \cdots \lessdot x_n$
and $0 \leq i \leq n$,
$\W{(x_j)_{i \leq j \leq n}}{*} \circ \W{(x_j)_{0 \leq j \leq i}}{*} = \W{x_\bullet}{*}$.
\end{lemma}

\begin{lemma} \label{lem:chain-indep-meet}
If $(\W{x}{y})_{x,y \in S, x \lessdot y}$ is a system of maps satisfying:
\begin{enumerate}
\item for any $x \lessdot y$ in $S$,
$\W{x}{y}$ is a partial bijection from source $L_x$ to target $L_y$, and
\item the $\W{\bullet}{\bullet}$ satisfy \ref{PS10-m},
\end{enumerate}
then:
\begin{enumerate}
\item[] For any two saturated chains
$x_0 \lessdot x_1 \lessdot \cdots \lessdot x_n$ and
$y_0 \lessdot y_1 \lessdot \cdots \lessdot y_m$
with $x_0 = y_0 = a$ and $x_n = y_m = b$,
$\W{x_\bullet}{*} = \W{y_\bullet}{*}$.
\end{enumerate}
\end{lemma}
\begin{proof} \disconnect
We prove the conclusion by induction on the interval $[a, b]$,
ordered by inclusion, starting with the smallest interval.
The base cases are $a = b$, in which case
$\W{x_\bullet}{*} = \W{y_\bullet}{*} = \id_{a}$.

If $a < b$ then $n, m \geq 1$, $x_1$ and $y_1$ exist, and
$a \lessdot x_1, y_1$.  If $x_1 = y_1$, by induction on $[x_1, b]$,
$\W{(x_i)_{1 \leq i \leq n}}{*} = \W{(y_i)_{1 \leq i \leq m}}{*}$,
and by lem.~\ref{lem:concat},
$\W{(x_i)_{0 \leq i \leq n}}{*} = \W{(y_i)_{0 \leq i \leq m}}{*}$.

If $x_1 \neq y_1$, then $x_1 \wedge y_1 = a$.
Applying antecedent (2),
there exist saturated chains
$x_1 = x^\prime_0 \lessdot x^\prime_1 \cdots \lessdot x^\prime_p = x_1 \vee y_1$
and
$y_1 = y^\prime_0 \lessdot y^\prime_1 \cdots \lessdot y^\prime_q = x_1 \vee y_1$
such that $\W{x_\bullet^\prime}{*} \circ \W{a}{x_1} =
\W{y_\bullet^\prime}{*} \circ \W{a}{y_1}$.
And since $x_1, y_1 \leq b$, $x_1 \vee y_1 \leq b$, there exists a
saturated chain
$x_1 \vee y_1 = z_0 \lessdot z_1 \lessdot \cdots \lessdot z_r = b$.
See fig.~\ref{fig:cim} for the configuration.
\begin{figure}[ht]
\hbox to \linewidth{\hfill%
\includegraphics[width=1in]{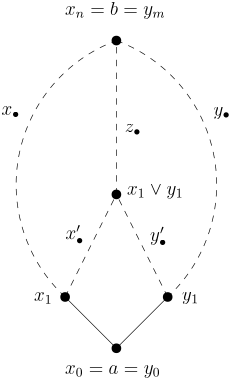}%
\hfill}
\caption{Configuration in the proof of lem.~\ref{lem:chain-indep-meet}}
\label{fig:cim}
\end{figure}

Then, repeatedly applying lem.~\ref{lem:concat}:
\begin{alignat*}{2}
\W{x_\bullet}{*}
& = \W{(x_i)_{1 \leq i \leq n}}{*} \circ \W{a}{x_1} \\
\interject{by induction on $[x_1, b]$,}
& = \W{z_\bullet}{*} \circ \W{x_\bullet^\prime}{*} \circ \W{a}{x_1} \\
\interject{by antecedent (2),}
& = \W{z_\bullet}{*} \circ \W{y_\bullet^\prime}{*} \circ \W{a}{y_1} \\
\interject{by induction on $[y_1, b]$,}
& = \W{(y_i)_{1 \leq i \leq m}}{*} \circ \W{a}{y_1} \\
& =\W{y_\bullet}{*}
\end{alignat*}
\end{proof}

Dually, we prove:

\begin{lemma} \label{lem:chain-indep-join}
If $(\W{x}{y})_{x,y \in S, x \lessdot y}$ is a system of maps satisfying:
\begin{enumerate}
\item for any $x \lessdot y$ in $S$,
$\W{x}{y}$ is a partial bijection from source $L_x$ to target $L_y$, and
\item the $\W{\bullet}{\bullet}$ satisfy \ref{PS10-j},
\end{enumerate}
then:
\begin{enumerate}
\item[] For any two saturated chains
$x_0 \lessdot x_1 \lessdot \cdots \lessdot x_n$ and
$y_0 \lessdot y_1 \lessdot \cdots \lessdot y_m$
with $x_0 = y_0 = a$ and $x_n = y_n = b$,
$\W{x_\bullet}{*} = \W{y_\bullet}{*}$.
\end{enumerate}
\end{lemma}

Thus, given the antecedents of either lem.~\ref{lem:chain-indep-meet} or
lem.~\ref{lem:chain-indep-join}, we can define $\P{x}{y}$ for any
$x \leq y$ in $S$:

\begin{definition} \label{def:phi-from-omega}
Given a system of maps $(\W{x}{y})_{x,y \in S, x \lessdot y}$
that satisfies \ref{PS10-},
we define, for any $x \leq y$ in $S$,
$\P{x}{y} = \W{s_\bullet}{*}$ for an arbitrarily chosen saturated chain
$x = s_0 \lessdot s_1 \lessdot \cdots \lessdot s_n = y$.
By lem.~\ref{lem:chain-indep-meet} or lem.~\ref{lem:chain-indep-join}, the
choice of $s_\bullet$ makes no difference.
\end{definition}

\begin{lemma} \label{lem:phi-from-omega}
Given a system of maps $(\W{x}{y})_{x,y \in S, x \lessdot y}$
that satisfies \ref{PS3-} and \ref{PS10-},
if $\P{x}{y}$ is not empty for some $x \leq y$ in $S$,
then
$\P{x}{y}$ is a partial bijection from $L_x$ to $L_y$,
$\dom \P{x}{y}$ is a filter of $L_x$,
$\im \P{x}{y}$ is an ideal of $L_y$,
and $\P{x}{y}$ is a lattice isomorphism.
For any $x \ S$, $\P{x}{x} = \id_{L_x}$.
\end{lemma}

\begin{lemma} \label{lem:phi-concat}
Given a system of maps $(\W{x}{y})_{x,y \in S, x \lessdot y}$
that satisfies \ref{PS10-},
and $x \leq y \leq z$ in $S$, then
$\P{x}{z} = \P{y}{z} \circ \P{x}{y}$.
\end{lemma}

\paragraph{Strengthening \texorpdfstring{\ref{PS11-mm} and \ref{PS11-jj}}%
{(PS7\9040\023\9042\047) and (PS8\9040\023\9042\050)}} \disconnect

\begin{lemma} \label{lem:join-from-cover}
Let $L$ be a finite lattice and $D$ be a down-set%
\footnote{A \emph{down-set} or \emph{poset ideal} of a poset $L$
is a subset $D \subset L$ where
if $x \in D$ and $y \leq x$ in $L$ then $y \in D$.}
of $L$.
Suppose that for any $a, b \in L$, $a, b \in D$ and
$a \wedge b \lessdot a, b$ imply that $a \vee b \in D$.
Then $a, b \in D$ implies that $a \vee b \in D$.
\end{lemma}
\begin{proof} \disconnect
We prove the consequent by induction on the pairs $(a, b)$, which we
order by the containment of the (finite) intervals
$[a \wedge b, a \vee b]$, with the smallest intervals first.

Assume $a, b \in D$.
Choose saturated chains
$a \wedge b = x_0 \lessdot x_1 \lessdot \cdots x_n = a$ and
$a \wedge b = y_0 \lessdot y_1 \lessdot \cdots y_m = b$.
If $n = 0$, then $a \wedge b = a$, so $a \leq b$ and
$a \vee b = b \in D$ by hypothesis.
Similarly, if $m = 0$, then $a \wedge b = b$, so $b \leq a$ and
$a \vee b = a \in D$ by hypothesis.

The remaining case is $n, m \geq 1$, so
$x_1$ and $y_1$ exist,
$a \wedge b \lessdot x_1 \leq a$, $x_1 \not\leq b$,
$a \wedge b \lessdot y_1 \leq b$, and $y_1 \not\leq a$.
See fig.~\ref{fig:jfc} for the configuration.
\begin{figure}[ht]
\hbox to \linewidth{\hfill%
\includegraphics[width=2in]{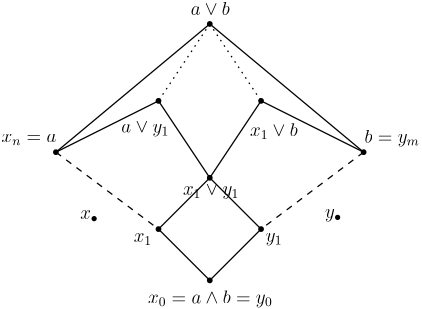}%
\hfill}
\caption{Configuration in the proof of lem.~\ref{lem:join-from-cover}}
\label{fig:jfc}
\end{figure}

Inductively apply the consequent to $a$ and $x_1 \vee y_1$, for
which $a \wedge (x_1 \vee y_1) \geq x_1 \gtrdot a \wedge b$ and
$a \vee (x_1 \vee y_1) = a \vee y_1 \leq a \vee b$,
so $[a \wedge (x_1 \vee y_1), a \vee y_1]$ is a proper subinterval of
$[a \wedge b, a \vee b]$.
Thus $a \vee y_1 = a \vee (x_1 \vee y_1) \in D$.

Similarly, inductively apply the consequent to $b$ and $x_1 \vee y_1$, for
which $b \wedge (x_1 \vee y_1) \geq y_1 \gtrdot a \wedge b$ and
$b \vee (x_1 \vee y_1) = b \vee x_1 \leq a \vee b$,
so $[b \wedge (x_1 \vee y_1), x \vee b]$ is a proper subinterval of
$[a \wedge b, a \vee b]$.
Thus $b \vee x_1 = b \vee (x_1 \vee y_1) \in D$.

Finally, inductively apply the consequent to $a \vee y_1$ and $b \vee x_1$,
for which
$(a \vee y_1) \wedge (b \vee x_1) \geq x_1 \vee y_1 > x_1 \wedge y_1$ and
$(a \vee y_1) \vee (b \vee y_1) = a \vee b$, so
$[(a \vee y_1) \wedge (b \vee x_1), a \vee b]$ is a proper subinterval of
$[a \wedge b, a \vee b]$.
Thus, $a \vee b \in D$.
\end{proof}

Dually, we prove:

\begin{lemma} \label{lem:meet-from-cover}
Let $L$ be a finite lattice and $U$ be an up-set%
\footnote{An \emph{up-set} or \emph{poset filter} of a poset $L$
is a set $U$ where
if $x \in U$ and $x \leq y$ in $L$ then $y \in U$.}
of $L$.
Suppose that for any $a, b \in L$, $a, b \in U$ and
$a, b \lessdot a \vee b$ imply that $a \wedge b \in U$.
Then $a, b \in U$ implies that $a \wedge b \in U$.
\end{lemma}

\begin{lemma} \label{lem:PS11-mm-stronger}
Given a system that satisfies the minimal axioms,
for any $x, y \in S$,
$\dom \P{x \wedge y}{x} \cap \dom \P{x \wedge y}{y} =
\dom \P{x \wedge y}{x \vee y}$.
\end{lemma}
\begin{proof} \disconnect
Define $L = [x \wedge y, x \vee y]$, which is a finite lattice
by \ref{PS1-}.
Define $D = \{ z \in L \mvert
\dom \P{x \wedge y}{x} \cap \dom \P{x \wedge y}{y} \subset
\dom \P{x \wedge y}{z} \}$.
By lem.~\ref{lem:concat}, if $x \wedge y \leq a \leq b \leq x \vee y$, then
$\dom \P{x \wedge y}{b} = \dom \P{a}{b} \circ \P{x \wedge y}{a}
\subset \dom \P{x \wedge y}{a}$,
so $D$ is a down-set of L.

Assume $a, b \in D$ and $a \wedge b \lessdot a, b$.
We prove $a \vee b \in D$:
\begin{alignat*}{2}
\interject{By \ref{PS11-mm} and def.~\ref{def:phi-from-omega},}
\dom \P{a \wedge b}{a} \cap \dom \P{a \wedge b}{b} & \subset
\dom \P{a \wedge b}{a \vee b} \\
\interject{composing all the maps with $\P{x \wedge y}{a \wedge b}$,}
\dom \P{a \wedge b}{a} \circ \P{x \wedge y}{a \wedge b} \cap
\dom \P{a \wedge b}{b} \circ \P{x \wedge y}{a \wedge b} & \subset
\dom \P{a \wedge b}{a \vee b} \circ \P{x \wedge y}{a \wedge b} \\
\interject{by lem.~\ref{lem:concat},}
\dom \P{x \wedge y}{a} \cap
\dom \P{a \wedge b}{b} & \subset
\dom \P{x \wedge y}{a \vee b} \\
\interject{since
$\dom \P{x \wedge y}{x} \cap \dom \P{x \wedge y}{y} \subset
\dom \P{x \wedge y}{a}$ and
$\dom \P{x \wedge y}{x} \cap \dom \P{x \wedge y}{y} \subset
\dom \P{x \wedge y}{b}$ by hypothesis,}
\dom \P{x \wedge y}{x} \cap \dom \P{x \wedge y}{y} & \subset
\dom \P{x \wedge y}{a} \cap \dom \P{a \wedge b}{b} \\
& \subset \dom \P{x \wedge y}{a \vee b}
\end{alignat*}
So $a \vee b \in D$.

Applying lem.~\ref{lem:join-from-cover} to $D$ shows that $D$ is
closed under joins.  Since $x, y \in D$ trivially, $x \vee y \in D$
and so
\begin{equation*}
\dom \P{x \wedge y}{x} \cap \dom \P{x \wedge y}{y} \subset
\dom \P{x \wedge y}{x \vee y}. \tag{a}
\end{equation*}

Following \citeHerr*{(7)}{(7)}\cite{Wor2025a}*{Lem.~4.6},
$\dom \P{x \wedge y}{x \vee y} =
\dom \P{x}{x \vee y} \circ \P{x \wedge y}{x} \subset
\dom \P{x \wedge y}{x}$
and similarly
$\dom \P{x \wedge y}{x \vee y} \subset
\dom \P{x \wedge y}{y}$,
so
$\dom \P{x \wedge y}{x \vee y} \subset
\dom \P{x \wedge y}{x} \cap \dom \P{x \wedge y}{y}$.
Combining that with (a) shows that
$\dom \P{x \wedge y}{x} \cap \dom \P{x \wedge y}{y} =
\dom \P{x \wedge y}{x \vee y}$.
\end{proof}

Dually, we prove:

\begin{lemma} \label{lem:PS11-jj-stronger}
Given a system that satisfies the minimal axioms,
for any $x, y \in S$,
$\im \P{x}{x \vee y} \cap \im \P{y}{x \vee y} =
\im \P{x \wedge y}{x \vee y}$.
\end{lemma}

\begin{lemma} \label{lem:im-dom-converse}
Given a system that satisfies the minimal axioms,
for any $x, y \in S$,
$\P{x \wedge y}{y} \circ \PI{x \wedge y}{x} =
\PI{y}{x \vee y} \circ \P{x}{x \vee y}$
as a partial bijection from $L_x$ to $L_y$.
\end{lemma}
\begin{proof} \disconnect
Assume $z \in \dom \P{x \wedge y}{y} \circ \PI{x \wedge y}{x}$.
Then $\PI{x \wedge y}{x}\,z \in \dom \P{x \wedge y}{y}$,
but since necessarily $\PI{x \wedge y}{x}\,z \in \dom \P{x \wedge y}{x}$, then
$\PI{x \wedge y}{x}\,z \in
\dom \P{x \wedge y}{x} \cap \dom \P{x \wedge y}{y}$ and
by lem.~\ref{lem:PS11-mm-stronger} and~\ref{lem:phi-concat},
$\PI{x \wedge y}{x}\,z \in \dom \P{x \wedge y}{x \vee y} =
\dom \P{x}{x \vee y} \circ \P{x \wedge y}{x} =
\dom \P{y}{x \vee y} \circ \P{x \wedge y}{y}$.
Then,
\begin{alignat*}{2}
(\P{x}{x \vee y} \circ \P{x \wedge y}{x})(\PI{x \wedge y}{x}\,z) & =
(\P{y}{x \vee y} \circ \P{x \wedge y}{y})(\PI{x \wedge y}{x}\,z) \\
\P{x}{x \vee y}(\P{x \wedge y}{x}(\PI{x \wedge y}{x}\,z)) & =
\P{y}{x \vee y}(\P{x \wedge y}{y}(\PI{x \wedge y}{x}\,z)) \\
\P{x}{x \vee y}\,z & =
\P{y}{x \vee y}(\P{x \wedge y}{y}(\PI{x \wedge y}{x}\,z)) \\
\PI{y}{x \vee y}(\P{x}{x \vee y}\,z) & =
\P{x \wedge y}{y}(\PI{x \wedge y}{x}\,z).
\end{alignat*}

In the reverse direction,
assume $z \in \dom \PI{y}{x \vee y} \circ \P{x}{x \vee y}$.
Then $\P{x}{x \vee y}\,z \in \dom \PI{y}{x \vee y} = \im \P{y}{x \vee y}$,
but since necessarily $\P{x}{x \vee y}\,z \in \im \P{x}{x \vee y}$, then
$\P{x}{x \vee y}\,z \in \im \P{x}{x \vee y} \cap \im \P{y}{x \vee y}$ and
by lem.~\ref{lem:PS11-jj-stronger} and~\ref{lem:phi-concat},
$\P{x}{x \vee y}\,z \in \im \P{x \wedge y}{x \vee y} =
\im \P{x}{x \vee y} \circ \P{x \wedge y}{x} =
\im \P{y}{x \vee y} \circ \P{x \wedge y}{y}$.
Then,
\begin{alignat*}{2}
(\P{x}{x \vee y} \circ \P{x \wedge y}{x})^{-1}(\P{x}{x \vee y}\,z) & =
(\P{y}{x \vee y} \circ \P{x \wedge y}{y})^{-1}(\P{x}{x \vee y}\,z) \\
\PI{x \wedge y}{x}(\PI{x}{x \vee y}(\P{x}{x \vee y}\,z)) & =
\PI{x \wedge y}{y}(\PI{y}{x \vee y}(\P{x}{x \vee y}\,z)) \\
\PI{x \wedge y}{x}\,z & =
\PI{x \wedge y}{y}(\PI{y}{x \vee y}(\P{x}{x \vee y}\,z)) \\
\P{x \wedge y}{y}(\PI{x \wedge y}{x}\,z) & =
\PI{y}{x \vee y}(\P{x}{x \vee y}\,z)
\end{alignat*}

Thus, $\P{x \wedge y}{y} \circ \PI{x \wedge y}{x} = \PI{y}{x \vee y} \circ \P{x}{x \vee y}$.
\end{proof}

\paragraph{Defining \texorpdfstring{$\relgamma$}{\83\263}} \disconnect

\begin{definition} \label{def:relxi}
Given a system that satisfies the minimal axioms,
for any $x \leq y$ in $S$,
we define the relationship $x \relxi y$ iff $\P{x}{y}$ is not empty.
\end{definition}

\begin{lemma} \label{lem:xi-join-1}
Given a system that satisfies the minimal axioms,
if $x \relxi x^\prime$ and $x \leq y$, then
$x \vee y \relxi x^\prime \vee y$.
\end{lemma}
\begin{proof} \disconnect
Choose a saturated chain
$x = z_0 \lessdot z_1 \lessdot \cdots \lessdot z_n = y$.
We will show by induction on $0 \leq i \leq n$ that
$z_i \relxi x^\prime \vee z_i$.

The base case is $i = 0$ and $z_0 = x$, which is true by hypothesis.

For $i \geq 0$,
if $z_i \relxi x^\prime \vee z_i$,
then $\dom \P{z_i}{x^\prime \vee z_i} \neq \zeroslash$,
so by lem.~\ref{lem:phi-from-omega},
$\hatone_{L_{z_i}} \in \dom \P{z_i}{x^\prime \vee z_i}$.
By \ref{PS3-}, $\dom \P{z_i}{z_{i+1}} \neq \zeroslash$,
so by lem.~\ref{lem:phi-from-omega},
$\hatone_{L_{z_i}} \in \dom \P{z_i}{z_{i+1}}$.
By lem.~\ref{lem:PS11-mm-stronger},
$\hatone_{L_{z_i}} \in \dom \P{z_i}{(x^\prime \vee z_i) \vee z_{i+1}}
= \dom \P{z_i}{x^\prime \vee z_{i+1}}
= \dom \P{z_{i+1}}{x^\prime \vee z_{i+1}} \circ \P{z_i}{z_{i+1}}$,
so $\P{z_{i+1}}{x^\prime \vee z_{i+1}}$ is not empty and
$z_{i+1} \relxi x^\prime \vee z_{i+1}$.

The case $i = n$ and $z_n = y$ gives the conclusion.
\end{proof}

\begin{lemma} \label{lem:xi-join-2}
Given a system that satisfies the minimal axioms,
if $x \relxi x^\prime$ and $y \relxi y^\prime$,
then $x \vee y \relxi x^\prime \vee y^\prime$.
\end{lemma}
\begin{proof} \disconnect
Applying lem.~\ref{lem:xi-join-1} to $x \relxi x^\prime$ and $x \vee y$,
we obtain $x \vee y \relxi x^\prime \vee y$.
Applying lem.~\ref{lem:xi-join-1} to $y \relxi y^\prime$ and $x \vee y$,
$x \vee y \relxi x \vee y^\prime$.
These imply that
$\hatone_{L_{x \vee y}} \in \dom \P{x \vee y}{x^\prime \vee y}$ and
$\hatone_{L_{x \vee y}} \in \dom \P{x \vee y}{x \vee y^\prime}$,
so by \ref{PS11-mm},
$\hatone_{L_{x \vee y}} \in \dom \P{x \vee y}{x^\prime \vee y^\prime}$
and thus $x \vee y \relxi x^\prime \vee y^\prime$.
\end{proof}

Dually, we prove:

\begin{lemma} \label{lem:xi-meet-2}
Given a system that satisfies the minimal axioms,
if $x \relxi x^\prime$ and $y \relxi y^\prime$,
then $x \wedge y \relxi x^\prime \wedge y^\prime$.
\end{lemma}

Thus, $\relxi$ is an ordered tolerance which can be extended to a
tolerance $\relgamma$:

\begin{definition} \label{def:gamma-from-xi}
Given a system that satisfies the minimal axioms,
for any $x, y \in S$, we define the relation
$x \relgamma y$ iff $x \wedge y \relxi x \vee y$.
\end{definition}

\begin{lemma} \label{lem:gamma-from-xi}
Given a system that satisfies the minimal axioms,
$\relxi$ is an ordered tolerance on $S$ and $\relgamma$ is a tolerance
on $S$.
\end{lemma}
\begin{proof} \disconnect
By def.~\ref{def:relxi} and~\ref{def:phi-from-omega},
$\relxi$ is reflexive.
That fact, together with lem.~\ref{lem:xi-meet-2},
lem.~\ref{lem:xi-join-2}, and def.~\ref{def:ord-tol} show that
$\relxi$ is an ordered tolerance.

Lem.~\ref{lem:ord-tol} then shows that $\relgamma$ is a tolerance.
\end{proof}

\paragraph{Equivalence} \disconnect

\begin{theorem} \label{th:axioms-min}
All alternatives of
the minimal axioms (def.~\forwardref{def:polytone-min}) are equivalent to the
baseline axioms (def.~\ref{def:polytone}).
\end{theorem}
\begin{proof} \disconnect
% (PS1-) immediate
% (PS2-) immediate
% (PS3-) immediate
% (PS10-) proved
% (PS11-) proved

Assume a system satisfies the baseline axioms.
Then minimal axioms \ref{PS1-}, \ref{PS2-}, and \ref{PS3-} are
immediately satisfied.

Define for $x \lessdot y$ in $S$, $\W{x}{y} = \P{x}{y}$.

Concerning \ref{PS10-m}:  Assume $x \wedge y \lessdot x, y$ in $S$.
Then by \ref{PS9} and lem.~\ref{lem:tol}, $x \wedge y \relgamma x, y$, so
$x \wedge y, x, y, \relgamma x \vee y$.
Then choose a saturated chain
$x = x_0 \lessdot x_1 \cdots \lessdot x_n = x \vee y$.
By \ref{PS6},
$\W{x_\bullet}{*} \circ \W{x \wedge y}{x} =
\P{x}{x \vee y} \circ \P{x \wedge y}{x} =
\P{x \wedge y}{x \vee y}$.
Similarly, chose a saturated chain
$y = y_0 \lessdot y_1 \cdots \lessdot y_m = x \vee y$.
By \ref{PS6},
$\W{y_\bullet}{*} \circ \W{y \wedge y}{y} = \P{x \wedge y}{x \vee y}$.
These show \ref{PS10-m} is satisfied.

Dually, \ref{PS10-j} holds in the system.

Concerning \ref{PS11-mm}:  Assume $x \wedge y \lessdot x, y$ in $S$ and
an arbitrary saturated chain
$x \wedge y =z_0 \lessdot z_1 \cdots \lessdot z_n = x \vee y$.
Again by \ref{PS9} and lem.~\ref{lem:tol},
$x$, $y$, $x \wedge y$, and $x \vee y$ are $\relgamma$ to each other.
By definition
$\W{x \wedge y}{x} = \P{x \wedge y}{x}$ and
$\W{x \wedge y}{y} = \P{x \wedge y}{y}$ and
by \ref{PS6},
$\W{z_\bullet}{*} = \P{x \wedge y}{x \vee y}$.
Then by \ref{PS8},
$\dom \P{x \wedge y}{x} \cap \dom \P{x \wedge y}{y} \subset
\dom \P{x \wedge y}{x \vee y}$.
Hence \ref{PS11-mm} is satisfied by the system.

Dually, \ref{PS11-jj} holds in the system.

Together, these show that the system satisfies the minimal axioms.

% \relgamma by lem:gamma-from-xi
% (PS1) immediate
% (PS2) immediate
% (PS3) lem:phi-from-omega
% (PS4) construction of phi
% (PS5) lem:phi-concat
% (PS6) construction of phi
% (PS7)
% (PS8)
% (PS9) PS3-, def:relxi

Conversely, assume a system satisfies the minimal axioms.  Then
baseline axioms \ref{PS1} and \ref{PS2} are immediately satisfied.

Since the system satisfies \ref{PS10-}, we can use
lem.~\ref{lem:chain-indep-meet}, lem.~\ref{lem:chain-indep-join}, and
def.~\ref{def:phi-from-omega} to define $\P{x}{y}$ for any
$x \leq y$ in $S$.  By construction, $\P{\bullet}{\bullet}$
satisfies \ref{PS4} and \ref{PS6}.

Define the relationship $\relgamma$ on $S$ by def.~\ref{def:relxi}
and~\ref{def:gamma-from-xi}.  By lem.~\ref{lem:gamma-from-xi},
$\relgamma$ is a tolerance on $S$.

Lem.~\ref{lem:phi-from-omega} shows that \ref{PS3} is satisfied.

Regarding \ref{PS5}:  If $x \leq_\gamma y \leq_\gamma z$ in $S$ and
$\im \P{x}{y} \cap \dom \P{y}{z} \neq \zeroslash$,
then by lem.~\ref{lem:phi-concat},
$\dom \P{x}{z} = \dom \P{y}{z} \circ \P{x}{y} \neq \zeroslash$,
so $x \relgamma y$.

Lem.~\ref{lem:phi-concat} shows that \ref{PS6} is satisfied.

Lem.~\ref{lem:PS11-jj-stronger} shows that \ref{PS7} is satisfied.
Similarly, lem.~\ref{lem:PS11-mm-stronger} shows that \ref{PS8} is satisfied.

Def.~\ref{def:relxi}, def.~\ref{def:gamma-from-xi}, and \ref{PS3-}
show that \ref{PS9} is satisfied.

Thus, the system satisfies the baseline axioms.
\end{proof}

\section{Maximal axioms}

Now we strengthen the set of statements we can adduce about a polytone
system.  Rather than showing our efforts step-by-step,
we will give the end result of our efforts as a definition,
then work through the proof that it is equivalent to the baseline
axioms.

The names of the minimal axioms are distinguished by including ``\suppl''.

\begin{definition} \label{def:polytone-max}
The \emph{maximal axioms}:
We define a \emph{polytone system} to be comprised of:
\begin{enumerate}
\item[\textbullet] a \emph{skeleton} lattice $S$,
\item[\textbullet] an \emph{overlap tolerance} $\relgamma$, which is
a tolerance on $S$,
\item[\textbullet] a family of \emph{blocks} $(L_x)_{x \in S}$, which
are lattices, and
\item[\textbullet] a family of \emph{connections}
$(\P{x}{y})_{x, y \in S, x \leq y}$, which are mappings,
\end{enumerate}
that satisfy these axioms:
\begin{enumerate}
\itemlabel{(PS1\suppl)}{PS1+} The skeleton $S$ is \lffc.
\itemlabel{(PS2\suppl)}{PS2+} The blocks $L_x$ are \fmc.
\itemlabel{(PS3\suppl)}{PS3+}
Each $\P{x}{y}$ is a partial bijection from $L_x$ to $L_y$.
If $x \not\relgamma y$, then $\P{x}{y}$ is empty.
If $x \relgamma y$, then
$\dom \P{x}{y}$ is a filter of $L_x$,
$\im \P{x}{y}$ is an ideal of $L_y$,
and $\P{x}{y}$ is a lattice isomorphism.
\itemlabel{(PS4\suppl)}{PS4+} For any $x \in S$, $\P{x}{x} = \id_{L_x}$.
\itemlabel{(PS5\suppl)}{PS5+} If $x \leq y \leq z$ in $S$ and
$\im \P{x}{y} \cap \dom \P{y}{z} \neq \zeroslash$, then $x \relgamma z$.
\itemlabel{(PS6\suppl)}{PS6+} For every $x \leq z \leq y$ in $S$,
then $\P{x}{y} = \P{z}{y} \circ \P{x}{z}$.
\itemlabel{(PS7\suppl)}{PS7+} For every $z \leq x \wedge y$ in $S$,
$\dom \P{z}{x} \cap \dom \P{z}{y} = \dom \P{z}{x \vee y}$.
\itemlabel{(PS8\suppl)}{PS8+} For every $x \vee y \leq z$ in $S$,
$\im \P{x}{z} \cap \im \P{y}{z} = \im \P{x \wedge y}{z}$.
\itemlabel{(PS9\suppl)}{PS9+} If $x \lessdot y$ in $S$, then
$x \relgamma y$ (and thus $\P{x}{y}$ is not empty).
\end{enumerate}
\end{definition}

Each maximal axiom (PS+$n$) is a stronger form of the corresponding
baseline axiom (PS$n$).
The additional axiom for a monotone system, (MS), is unchanged.

We copy this lemma from \cite{Wor2025a}*{Lem.~4.7}:

\begin{lemma} \label{lem:MCd-stronger}
If $z \leq_\gamma x, y$ in $S$, then
$\dom \P{z}{x} \cap \dom \P{z}{y} = \dom \P{z}{x \vee y}$.
If $x, y \leq_\gamma z$ in $S$, then
$\im \P{x}{z} \cap \im \P{y}{z} = \im \P{x \wedge y}{z}$.
\end{lemma}

\begin{lemma} \label{lem:phi-extended}
In a polytone system that satisfies the baseline axioms
(def.~\ref{def:polytone}), with the extension of
$\P{\bullet}{\bullet}$ (def.~\ref{def:phi-extended}):
\begin{enumerate}
\item If $x \leq y \leq z$ in $S$ and
$\im \P{x}{y} \cap \dom \P{y}{z} \neq \zeroslash$, then $x \relgamma z$.
\item For every $x \leq z \leq y$ in $S$,
then $\P{x}{y} = \P{z}{y} \circ \P{x}{z}$.
\item For every $z \leq x \wedge y$ in $S$,
$\dom \P{z}{x} \cap \dom \P{z}{y} = \dom \P{z}{x \vee y}$.
\item For every $x \vee y \leq z$ in $S$,
$\im \P{x}{z} \cap \im \P{y}{z} = \im \P{x \wedge y}{z}$.
\end{enumerate}
\end{lemma}
\begin{proof} \disconnect
Regarding (1):
Since $\im \P{x}{y} \cap \dom \P{y}{z} \neq \zeroslash$,
$\im \P{x}{y} \neq \zeroslash$, so $x \leq_\gamma y$.
Similarly, $\dom \P{y}{z} \neq \zeroslash$, so $y \leq_\gamma z$.
Thus, $x \leq_\gamma y \leq_\gamma z$.
Applying \ref{PS5} shows that $x \relgamma z$.

Regarding (2): We prove this by cases:
\begin{enumerate}
\item $x \not\relgamma z$: $\P{x}{z}$ is empty, and
by lem.~\ref{lem:tol}(\ref{lem:tol:i1}),
$x \not\relgamma y$, so $\P{x}{y}$ is empty.
Thus $\P{x}{y} = \P{z}{y} \circ \P{x}{z}$.
\item $z \not\relgamma y$:
Similarly, $\P{z}{y}$ is empty, $x \not\relgamma y$, and
$\P{x}{y}$ is empty, so
$\P{x}{y} = \P{z}{y} \circ \P{x}{z}$.
\item $x \relgamma z$, $z \relgamma y$, and $x \relgamma y$:
By \ref{PS6}, $\P{x}{y} = \P{z}{y} \circ \P{x}{z}$.
\item $x \relgamma z$, $z \relgamma y$, but $x \not\relgamma y$:
By def.~\ref{def:phi-extended}, $\P{x}{y}$ is empty.
By \ref{PS5}, $\im \P{x}{y} \cap \dom \P{y}{z} = \zeroslash$,
which implies $\P{y}{z} \circ \P{x}{y}$ is empty.
\end{enumerate}

Regarding (3): We prove this by cases:
\begin{enumerate}
\item $z \relgamma x$ and $z \relgamma y$: The conclusion follows by
lem.~\ref{lem:MCd-stronger}.
\item $z \not\relgamma x$:
By lem.~\ref{lem:tol}(\ref{lem:tol:i1}),
$z \not\relgamma x \vee y$, so
$\dom \P{z}{x} = \dom \P{z}{x \vee y} = \zeroslash$,
which proves the conclusion.
\item $z \not\relgamma x$: We prove the conclusion similarly to case (2).
\end{enumerate}

Regarding (4): We prove this dually to (3).
\end{proof}

\begin{theorem} \label{th:axioms-max}
The maximal axioms (def.~\forwardref{def:polytone-max}) are equivalent to the
baseline axioms (def.~\ref{def:polytone}).
\end{theorem}
\begin{proof} \disconnect
The maximal axioms imply the baseline axioms, since each of the
baseline axioms is an immediate consequence of the corresponding
maximal axiom.

Assume a system satisfies the baseline axioms.  Clearly,
$\P{x}{x}$ can be extended to all $x \leq y$ by applying
def.~\ref{def:phi-extended}.
Then axioms \ref{PS1+}, \ref{PS2+}, and \ref{PS4+} are satisfied
because they are identical to
axioms \ref{PS1}, \ref{PS2}, and \ref{PS4}
Axioms \ref{PS3+} and \ref{PS9+} are satisfied because of axioms
\ref{PS3}, \ref{PS9}, and the fact that we've constructed
$\P{x}{y}$ to be empty iff $x \not\relgamma y$.
Lem.~\ref{lem:phi-extended} shows
axioms \ref{PS5+}, \ref{PS6+}, \ref{PS7+}, and \ref{PS8+} are
satisfied.
\end{proof}

% Push the e-mail address down a bit.
\vspace{3em}

\end{document}